\documentclass[10pt, english]{amsart}

\usepackage{amsmath,amssymb,enumerate,bbm}

\usepackage[T1]{fontenc}
\usepackage[utf8]{inputenc}
\usepackage[all]{xy}

\usepackage{babel}
\usepackage{amstext}
\usepackage{amsmath}
\usepackage{amsfonts}
\usepackage{latexsym}
\usepackage{ifthen}

\usepackage{hyperref}

\usepackage{xypic}
\xyoption{all}
\newcommand{\rk}{{\rm rk}}

\newtheorem{lemma1}{}[section]

\newenvironment{lemma}{\begin{lemma1}{\bf Lemma.}}{\end{lemma1}}
\newenvironment{example}{\begin{lemma1}{\bf Example.}\rm}{\end{lemma1}}

\newenvironment{theorem}{\begin{lemma1}{\bf Theorem.}}{\end{lemma1}}
\newenvironment{proposition}{\begin{lemma1}{\bf Proposition.}}{\end{lemma1}}
\newenvironment{corollary}{\begin{lemma1}{\bf Corollary.}}{\end{lemma1}}
\newenvironment{remark}{\begin{lemma1}{\bf Remark.}\rm}{\end{lemma1}}
\newenvironment{remarks}{\begin{lemma1}{\bf Remarks.}\rm}{\end{lemma1}}
\newenvironment{definition}{\begin{lemma1}{\bf Definition.}}{\end{lemma1}}

\newenvironment{conjecture}{\begin {lemma1}{\bf Conjecture.}}{\end{lemma1}}

\newenvironment{remark*}{{\bf Remark.}}{}
\newenvironment{example*}{{\bf Example.}}{}

\newcommand{\R}{\ensuremath{\mathbb{R}}}
\newcommand{\Q}{\ensuremath{\mathbb{Q}}}
\newcommand{\Z}{\ensuremath{\mathbb{Z}}}
\newcommand{\C}{\ensuremath{\mathbb{C}}}
\newcommand{\N}{\ensuremath{\mathbb{N}}}
\newcommand{\PP}{\ensuremath{\mathbb{P}}}

\newcommand{\merom}[3]{\ensuremath{#1:#2 \dashrightarrow #3}}

\newcommand{\holom}[3]{\ensuremath{#1:#2  \rightarrow #3}}
\newcommand{\fibre}[2]{\ensuremath{#1^{-1} (#2)}}

\makeatletter
\ifnum\@ptsize=0  \fi \ifnum\@ptsize=2  \fi 

\newcommand\sE{{\mathcal E}}

\newcommand\sF{{\mathcal F}}
\newcommand\sG{{\mathcal G}}

\newcommand\sO{{\mathcal O}}

\DeclareMathOperator*{\pic}{Pic}

\DeclareMathOperator*{\Pic0}{Pic^0}

\DeclareMathOperator*{\nons}{nons}

\newcommand{\chow}[1]{\ensuremath{\mathcal{C}(#1)}}

\title{On a conjecture of Beltrametti and Sommese} 
\date{November 30, 2017}

\author{Andreas H\"oring}


\address{Andreas H\"oring, Universit\'e Pierre et Marie Curie, Institut de Math\'ematiques de Jussieu, Equipe de Topologie et G\'eom\'etrie Alg\'ebrique, 175, rue du Chevaleret, 75013 Paris, France}
\email{hoering@math.jussieu.fr}

\begin{document}

\begin{abstract} 
Let $X$ be a projective manifold of dimension $n$.
Beltrametti and Sommese conjectured that if $A$ is an ample
divisor such that $K_X+(n-1)A$ is nef, then $K_X+(n-1)A$ has non-zero global sections.
We prove a weak version of this conjecture in arbitrary dimension.
In dimension three, we prove the stronger non-vanishing conjecture
of Ambro, Ionescu and Kawamata and give an application to Seshadri constants.
\end{abstract}

\maketitle


\section{Introduction}

\subsection{The main result}

The aim of this paper is to study the following effective non-vanishing conjecture, due to 
Beltrametti and Sommese \cite[Conj. 7.2.7]{BS95}. 

\begin{conjecture} \label{conjectureBS} 
Let $X$ be a projective manifold of dimension $n$, 
and let $A$ be an ample Cartier divisor such that $K_X+(n-1)A$ is nef.
Then  we have
$$
H^0(X, \sO_X(K_X+(n-1)A)) \neq 0.
$$
\end{conjecture}

By the classification of Fujita and Ionescu \cite{Ion86, Fuj87} the adjoint divisor $K_X+(n-1)A$ is nef 
unless we are in a very special situation ($X$ is a projective space, quadric etc.),
so the conjecture applies to adjoint linear systems on almost every variety.  
If $X$ is a surface it is an immediate consequence
of the Riemann-Roch formula and classical results on surfaces, but in higher dimension the situation is much
more complicated.  
Conjecture \ref{conjectureBS} and its (conjectural) generalisation due to Ambro \cite{Amb99}, Ionescu \cite{Cet90} and Kawamata \cite{Kaw00}
have been studied by several authors during the last years 
\cite{Kaw00}, \cite{CCZ05}, \cite{Xie05}, \cite{Fu06}, \cite{Fu07}, \cite{Bro09}, \cite{a6}.
We prove a weak version of the Beltrametti-Sommese conjecture in arbitrary dimension:

\begin{theorem} \label{theoremBS}
Let $X$ be a normal, projective variety of dimension $n \geq 2$ with at most rational singularities, and let
$A$ be a nef and big Cartier divisor on $X$ such that 
$K_X+(n-1)A$ is generically nef (cf. Definition \ref{definitiongenericallynef}). 
Then there exists a $j \in \{1, \ldots, n-1\}$ such that 
\[
H^0(X, \sO_X(K_X+jA)) \neq 0.
\]
In particular if $A$ is effective, then 
\[
H^0(X, \sO_X(K_X+(n-1)A)) \neq 0.
\]
If $X$ has irrational singularities, the statement still holds unless $(X,A)$ is birationally a scroll (cf. Definition \ref{definitionbirationalscroll}) over a curve
of positive genus.
\end{theorem}

Note that the conclusion of our theorem is {\em a priori}\footnote{It is an open problem due to Tsuji  \cite{Tsu04} whether 
$h^0(X, \sO_X(K_X+jA)) \leq h^0(X, \sO_X(K_X+(j+1)A))$ holds without assuming $A$ effective.} weaker than Conjecture \ref{conjectureBS}, but it should be equally useful for applications.

\subsection{The technique}

Let $X$ be a projective manifold of dimension $n$, and let $A$ be a nef and big Cartier divisor on $X$.   
By the Kawamata-Viehweg vanishing theorem one has
\[
\chi(X, \sO_X(K_X+tA))  =  h^0(X, \sO_X(K_X+tA)) \qquad \forall \ t \in \N, 
\]
so the non-vanishing problem reduces to studying the Hilbert polynomial $\chi(X, \sO_X(K_X+tA))$.  
By Serre duality 
\[
\chi(X, \sO_X(K_X+tA)) = (-1)^n \chi(X, \sO_X(-tA))
\]
is a polynomial of degree $n$ in $t$ which can be computed by the Riemann-Roch formula
\[
\chi(X, \sO_X(-tA)) = [ch(-tA) \cdot  td(T_X)]_n,
\]
where $[ \  ]_n$ denotes the component of degree $n$ in $A(X) \otimes \Q$.
Using the formulae
\[
ch(-tA) = \sum_{k=0}^n \frac{(-tA)^k}{k!}
\]
for the Chern character and
\[
td(T_X) = 1 - \frac{1}{2} K_X + \frac{1}{12} (K_X^2+c_2(X)) + \ldots + \chi(X, \sO_X)
\]
for the Todd class of $T_X$, we see that $\chi(X, \sO_X(K_X+tA))$ equals
\begin{equation} \label{RR}
\frac{A^n}{n!} t^n + 
\frac{A^{n-1} \cdot  K_X}{2(n-1)!} t^{n-1} + 
\frac{A^{n-2} \cdot  (K_X^2+c_2(X))}{12(n-2)!}t^{n-2} +
\ldots + (-1)^n \chi(X, \sO_X). 
\end{equation}

The idea of the proof of Theorem \ref{theoremBS} is now as follows: 
we argue by contradiction and suppose that for all $j \in \{ 1, \ldots, n-1 \}$ we have $h^0(X, \sO_X(K_X+jA)) = 0$.
Thus $1, \ldots, n-1$ are roots of the Hilbert polynomial $\chi(X, \sO_X(K_X+tA))$.
It is an undergraduate exercise to translate the assumption 
into equations involving the coefficients of the Hilbert polynomial above (cf. Lemma \ref{lemmaminusone}),
but it significantly simplifies the problem by reducing it to
controlling the characteristic classes $K_X, c_2(X)$ and $\chi(X, \sO_X)$.
Somewhat surprisingly this immediately allows us to deal with the case where
$X$ is rationally connected (so $\chi(X, \sO_X)=1$).
If $X$ is a minimal model Miyaoka's theorem tells us that the second Chern class $c_2(X)$
is pseudoeffective which is largely sufficient to conclude. 
More generally if $X$ is not uniruled we still know that $\Omega_X$
is generically nef, so a twisted version of Miyaoka's statement \cite[Thm.2.1]{Fuk05}, \cite[Cor.3.13]{a6} 
still allows us to control the second Chern class.
The most delicate case is thus when $X$ is uniruled but not rationally connected.
In particular the case of birational scrolls will need some additional effort.

\begin{definition} \label{definitionbirationalscroll}
Let $X$ be a normal, projective variety, and let $A$ be a nef and big Cartier divisor on $X$.
We say that $(X, A)$ is birationally a scroll if there exists
a birational morphism \holom{\mu}{X'}{X} from a projective manifold $X'$
and a fibration \holom{\varphi}{X'}{Y} onto a projective manifold $Y$
such that the general fibre $F$ admits a birational morphism $\tau: F \rightarrow \PP^{n-m}$
and $\sO_F(\mu^* A) \simeq \tau^* \sO_{\PP^{n-m}}(1)$.
\end{definition}

Using the foliated Mori theory due to Miyaoka and Bogomolov-McQuillan we prove the following:

\begin{theorem}\label{theoremgenericnefCartier}
Let $X$ be a normal, projective variety of dimension $n$.
Let $A$ be a nef and big Cartier divisor on $X$. 
If $(X, A)$ is not birationally a scroll, then $\Omega_X\hspace{-0.8ex}<\hspace{-0.8ex}A\hspace{-0.8ex}>$ is generically nef.
\end{theorem}

This theorem can be seen as a foliated version of the well-known statement that
if $X$ is a projective manifold and $A$ is ample, then $K_X+nA$ is nef unless $X \simeq \PP^n$
and $A \simeq \sO_{\PP^n}(1)$. Note that if $(X, A)$ is birationally a scroll, then
$\Omega_X\hspace{-0.8ex}<\hspace{-0.8ex}A\hspace{-0.8ex}>$ is not generically nef even if we
assume that $\det (\Omega_X\hspace{-0.8ex}<\hspace{-0.8ex}A\hspace{-0.8ex}>)$ is generically nef.
Indeed for $n \geq 2$,
set $X:=Y \times \PP^{n-m}$ where $Y$ is a projective manifold of dimension $1 \leq m \leq n-1$ with nef canonical divisor. 
Let $A_Y$ be an ample Cartier divisor on $Y$ and $H$ be the hyperplane divisor on  $\PP^{n-m}$, 
then $A:=p_Y^* A_Y+p_{\PP^{n-m}}^*H$ is ample and $K_X+nA$ is nef. Nevertheless the twisted bundle
$\Omega_X\hspace{-0.8ex}<\hspace{-0.8ex}A\hspace{-0.8ex}>$ is not generically nef: 
this would imply that $\Omega_{X/Y}\hspace{-0.8ex}<\hspace{-0.8ex}A\hspace{-0.8ex}>$
is generically nef, yet even its determinant 
$$
K_{X/Y}+(n-m)A = 
(n-m) p_Y^* A_Y - p_{\PP^{n-m}}^*H
$$ 
is not generically nef. 

\subsection{Generalisations and applications}

In dimension three, the techniques developed for the proof of Theorem \ref{theoremBS} 
give an affirmative answer for the stronger non-vanishing conjecture of
Ambro, Ionescu and Kawamata which so far is only known in rather special cases.

\begin{theorem} \label{theoremIK}
Let $X$ be a normal, projective threefold with at most $\Q$-factorial canonical singularities, 
and let $A$ be a nef and big Cartier divisor on $X$
such that $K_X+A$ is nef. Then we have
\[
H^0(X, \sO_X(K_X+A)) \neq 0. 
\]
\end{theorem}

Note that while $K_X$ is only supposed to be $\Q$-Cartier, it is crucial for our proof to suppose that $A$ is Cartier.
In fact the statement is false if $A$ is a Weil divisor which is merely $\Q$-Cartier:
Iano-Fletcher has constructed an example \cite[Ex.16.1]{IF00} of a $\Q$-Fano threefold $X$ of index $30$ with terminal singularities 
such that
$$
H^0(X,  \sO_X(-K_X)) = H^0(X, \sO_X(K_X+ 2(-K_X)))=0.
$$
As A. Broustet pointed out to me, a desingularisation of the threefold in the Iano-Fletcher example 
gives an example of a smooth projective threefold $X$
and $A$ a big Cartier divisor on $X$ such that $K_{X}+A$ is pseudoeffective but not effective. Thus our statement
is almost optimal. A direct consequence of Theorem \ref{theoremIK} is the following special case of Lazarsfeld's conjecture on Seshadri constants
(cf. \cite[Lemme 4.11]{Bro09}).

\begin{theorem} 
Let $X$ be a normal, projective threefold with at most $\Q$-factorial canonical singularities, 
and let $A$ be a nef and big Cartier divisor on $X$
such that $K_X+A$ is nef and big. Then we have
\[
\varepsilon(K_X+A, x) \geq 1
\]
for every $x \in X$ sufficiently general.

In particular if the anticanonical divisor of $X$ is nef and $L$ is a nef and big Cartier divisor on $X$, then
\[
\varepsilon(L, x) \geq 1
\]
for every $x \in X$ sufficiently general.
\end{theorem}

{\bf Acknowledgements.} I would like to thank A. Broustet, S. Boucksom and T. de Fernex for
discussions on various questions related to this paper. I would like to thank the referee for
pointing out a number of rather serious mistakes in the initial version of this paper.

\section{Notation and basic material}

We work over the complex numbers, topological notions always refer to the Zariski topology.
For general definitions we refer to Hartshorne's book \cite{Har77}.
We will frequently use standard terminology and results 
of the minimal model program (MMP) as explained in \cite{KM98} or \cite{Deb01}.

A variety is an integral scheme of finite type over $\C$, a manifold is a smooth variety.
A fibration is a proper, surjective 
morphism \holom{\varphi}{X}{Y} between normal varieties such that
$\dim X>\dim Y$ and $\varphi_* \sO_X \simeq \sO_Y$,
that is all the fibres are connected. Fibres are always scheme-theoretic fibres.
Points are always supposed to be closed.

Let $X$ be a normal variety. The singular locus of $X$ has codimension at least two,
so we have an isomorphism $Cl(X_{\nons}) \rightarrow Cl(X)$. 
We define the canonical divisor $K_X \in Cl(X)$ as the image of $\det T_{X_{\nons}}$.
Let \holom{\varphi}{X}{Y} be a fibration between projective manifolds.
We set
\[
K_{X/Y} := K_X - \varphi^* K_Y
\]
for the relative canonical divisor. 

A property (smoothness, local freeness, etc.) depending on a point $x \in X$ holds in codimension $k$,
if there exists a closed subset $Z \subset X$ of codimension bigger than $k$ such that
the property holds for every $x \in X \setminus Z$.

Let $X$ be a normal, projective variety.
A $\Q$-divisor will always be a $\Q$-Weil divisor, not necessarily $\Q$-Cartier. We will frequently use that
on a normal variety $X$, there is a bijection between Weil divisors $D$ and reflexive sheaves of rank one $\sO_X(D)$
\cite[App., Thm. 3]{Rei79}.

Let $X$ be  a normal, projective variety.
For every $k \in \{ 0, \ldots, \dim X\}$ we denote by $A_k(X)$ the group of $k$-dimensional cycles modulo rational equivalence,
and by $\pic(X)$ the group of isomorphism classes of line bundles.  
We denote by
\[
\pic(X)^k \times A_k(X) \rightarrow \Z, \ (D_1, \ldots, D_k, [Z]) \ \mapsto D_1 \cdots D_k \cdot [Z] 
\]
the intersection product as defined in \cite[Ch.2]{Ful84}. More generally if we consider Cartier divisors and 
cycles with coefficients in $\Q$, we get a pairing with values in $\Q$
which we often abbreviate by
\[
D_1 \cdots D_k \cdot [Z] =: D_1 \cdots D_k \cdot Z.
\]
Suppose now that $X$ is a normal, projective variety of dimension $n$ that is smooth in codimension two.
Then we have an isomorphism $A_{n-2}(X_{\nons}) \rightarrow A_{n-2}(X)$, 
so if $E$ is a coherent sheaf on $X$, we define
$c_2(E) \in A_{n-2}(X)$ as the image of $c_2(E|_{X_{\nons}})$ under this isomorphism.  
In particular we define the second Chern class $c_2(X)$ as the image of $c_2(T_{X_{\nons}})$.

We denote by $N^1(X)_\R$ the vector space of $\R$-Cartier divisors modulo numerical equivalence,
and by $N_1(X)_\R$ its dual, the space of 1-cycles modulo numerical equivalence.
A divisor class $\alpha \in N^1(X)_\R$ is pseudoeffective if it is in the closure of the cone of effective divisors in $N^1(X)_\R$.
By \cite{BDPP04} this is equivalent to
\[
\alpha \cdot C \geq 0
\]
for every $C$ a member of a covering family of curves for $X$.

Birationally, every projective manifold admits a fibration that separates the rationally connected part and the non-uniruled part: the 
{\it MRC-fibration} or {\it rationally connected quotient}:

\begin{theorem} \cite{Cam92}, \cite{GHS03}, \cite{KMM92}
\label{theoremMRC}
Let $X$ be a uniruled, projective manifold. Then there exists a projective manifold $X'$,
a birational morphism \holom{\mu}{X'}{X} and a fibration 
\holom{\varphi}{X'}{Y} onto a  projective manifold $Y$ 
such that the general fibre is rationally connected and the variety $Y$ is 
not uniruled. 
\end{theorem}

\begin{remarks} \label{remarksRCquotient} 
\begin{enumerate}
\item We call $Y$ the base of the MRC-fibration. This is a slight abuse of language since the MRC-fibration is 
only unique up to birational equivalence of fibrations (cf. \cite{Cam04b}). Since the dimension of $Y$ does not depend
on the birational model, it still makes sense to speak of the dimension of
the base of the MRC-fibration. 
\item If $X$ is a normal, projective variety, we define the MRC-fibration of $X$ to be the MRC-fibration of
some desingularisation $X' \rightarrow X$. 
We say that a normal variety is rationally connected if $X'$ is
rationally connected. Note that with this definition a cone over an elliptic curve is not rationally connected
(it is merely rationally chain-connected).  
\item The MRC-fibration is almost regular, i.e. there exist open dense sets $X_0 \subset X$ and $Y_0 \subset Y$ such that
the restriction of the rational map $\merom{\varphi}{X}{Y}$ to $X_0$ gives a regular (proper) fibration
$\holom{\varphi|_{X_0}}{X_0}{Y_0}$. In particular we can see the general $\varphi$-fibre as a submanifold of $X$. 
Note also that if $Y$ has dimension one, the almost regular map $\varphi$ is regular.
\end{enumerate}
\end{remarks}

\subsection{$\Q$-twisted sheaves and generic nefness}

We adapt the notion of $\Q$-twisted vector bundles \cite[Ch.6.2]{Laz04b} to our setting.

\begin{definition} \label{definitiontwistedsheaf} \cite{Miy87}
Let $X$ be a normal, projective variety. A $\Q$-twisted sheaf 
\[
\sF\hspace{-0.8ex}<\hspace{-0.8ex}\delta\hspace{-0.8ex}>
\]
is an ordered pair consisting of a coherent sheaf $\sF$  
and a numerical equivalence class $\delta \in N^1(X)_\Q$. 
The $\Q$-twisted sheaf $\sF\hspace{-0.8ex}<\hspace{-0.8ex}\delta\hspace{-0.8ex}>$ is torsion-free
if $\sF$ is torsion-free.
If $A$ is 
$\Q$-Cartier $\Q$-divisor on $X$ we write $\sF\hspace{-0.8ex}<\hspace{-0.8ex}A\hspace{-0.8ex}>$ for the twist of $\sF$ by the numerical class of $A$.
\end{definition}

In the situations we are interested in $\sF$ will either be a torsion-free sheaf on a normal variety
or the sheaf of K\"ahler differentials of a normal variety. 

\begin{definition} \label{definitionmrgeneral}
Let $X$ be a normal, projective variety, and let 
$H_1, \ldots, H_{n-1}$ be a collection of ample Cartier divisors.
A MR-general curve $C \subset X$ is an intersection
\[
D_1 \cap \ldots \cap D_{n-1}
\]
for general $D_j \in | m_j H_j |$ where $m_j \gg 0$. 
\end{definition}

\begin{remark} \label{remarkmrgeneral}
The abbreviation MR stands of course for Mehta-Ramanathan, alluding to the well-known fact \cite{MR82}
that the Harder-Narasimhan filtration of a torsion-free sheaf commutes with restriction to a MR-general curve. 
\end{remark}

Let $X$ be a normal, projective variety, and let  $\sF$ be a coherent sheaf that is locally free in codimension one.
A MR-general curve $C$ is contained in the open set 
where $\sF$ is locally free. Thus  $\sF|_C:=\sF \otimes \sO_C$ is a vector bundle and the following definition makes sense.

\begin{definition} \label{definitiongenericallynef}
Let $X$ be a normal, projective variety of dimension $n$, and let 
$\sF$   be a coherent sheaf on $X$ that is locally free in codimension one.
The $\Q$-twisted sheaf $\sF\hspace{-0.8ex}<\hspace{-0.8ex}\delta\hspace{-0.8ex}>$ 
is generically nef if its restriction to every MR-general curve $C$ is a nef $\Q$-vector bundle 
in the sense of \cite[Defn. 6.2.3]{Laz04b}, i.e.
\[
c_1(\sO_{\PP(\sF|_C)}(1))+\pi^* \delta
\]
is a nef class in $N^1(\PP(\sF|_C))_\R$, where $\sO_{\PP(\sF|_C)}(1)$ is the tautological line bundle
on the projectivised vector bundle $\PP(\sF|_C)$.

A $\Q$-divisor $D$ on $X$ is generically nef if 
\[
D \cdot H_1 \cdots H_{n-1} \geq 0
\]
for any collection of ample Cartier divisors $H_1, \ldots, H_{n-1}$.
\end{definition}

\begin{remarks}
a) An effective $\Q$-divisor $D$ is generically nef. 

b) A $\Q$-divisor $D$ on $X$ is generically nef if for $m\gg0$ sufficiently divisible the reflexive sheaf $\sO_X(mD)$ is generically nef.

c) If $D$ is a pseudoeffective $\Q$-Cartier $\Q$-divisor, it is generically nef.
\end{remarks}

For lack of reference we collect some basic properties of generically nef sheaves. The proof is elementary and left to the reader.

\begin{lemma} \label{lemmagenericallynef}
\begin{enumerate}
\item Let $X$ be a normal, projective variety of dimension $n$, and let 
$\sF$   be a coherent sheaf on $X$ that is locally free in codimension one.
The $\Q$-twisted sheaf $\sF\hspace{-0.8ex}<\hspace{-0.8ex}\delta\hspace{-0.8ex}>$  is generically nef if and only if its bidual $\sF^{**}\hspace{-0.8ex}<\hspace{-0.8ex}\delta\hspace{-0.8ex}>$
is generically nef.
\item A $\Q$-divisor $D$  on a normal, projective variety $X$ is generically nef if and only if 
\[
D \cdot H_1 \cdots H_{n-1} \geq 0
\]
for any collection of nef Cartier divisors $H_1, \ldots, H_{n-1}$.
\item Let \holom{\mu}{X'}{X} be a birational morphism between normal varieties, and
let $A$ be $\Q$-Cartier $\Q$-divisor on $X$. Let $C \subset X$ be a MR-general curve\footnote{A MR-general curve
does not meet the image of the exceptional locus, so we can consider it also as a curve in $X'$.}, then
\[
(K_{X'}+\mu^*A) \cdot C =  (K_{X}+A) \cdot C.
\]
If $K_X+A$ is not generically nef, then $K_{X'}+\mu^* A$ is not generically nef.
\end{enumerate}
\end{lemma}

Let $X$ be a projective variety that is smooth in codimension two, and let $D$ be a $\Q$-divisor on $X$.
Let $S \subset X$ be a surface that is a complete intersection of general very ample divisors. Then $S$ is not contained
$\mbox{Supp}(D)$, so the restriction $D|_S$ is well-defined. Moreover $S$ is smooth, so $D|_S$ is $\Q$-Cartier
and the following definition makes sense. 

\begin{definition} \label{definitionnefcodimone}
Let $X$ be a normal, projective variety of dimension $n$ that is smooth in codimension two, and let $D$ be a $\Q$-divisor on $X$.
We say that $D$ is nef in codimension one if for every collection
$H_1, \ldots, H_{n-2}$ of ample Cartier divisors and
$S \subset X$ a complete intersection
\[
D_1 \cap \ldots \cap D_{n-2}
\]
of general $D_j \in | m_j H_j |$ where $m_j \gg 0$, the restriction $D|_S$ is nef. 
\end{definition}

\begin{lemma}\footnote{This statement seems now a bit too optimistic to me, see 
\url{http://math.unice.fr/\~hoering/articles/remark-bs-conjecture.pdf} for a correction.} \label{lemmamiyaokaQ}
Let $X$ be a normal, projective variety of dimension $n \geq 2$ that is smooth in codimension two.
Let $E$ be a reflexive sheaf over $X$ such that $\det E$ is $\Q$-Cartier,
and $\delta$ a numerical equivalence class in $N^1(X)_\Q$.
If $E\hspace{-0.8ex}<\hspace{-0.8ex}\delta\hspace{-0.8ex}>$ is generically 
nef and $c_1( E\hspace{-0.8ex}<\hspace{-0.8ex}\delta\hspace{-0.8ex}>)$ is nef in codimension one, then
$$
H_1 \cdots H_{n-2} \cdot c_2(E\hspace{-0.8ex}<\hspace{-0.8ex}\delta\hspace{-0.8ex}>) \geq 0,
$$
where $H_1, \ldots, H_{n-2}$ is a collection of ample Cartier divisors on $X$. 
\end{lemma}

Recall that the isomorphism $A_{n-2}(X_{\nons}) \rightarrow A_{n-2}(X)$ allows to define the second Chern class
$c_2(E\hspace{-0.8ex}<\hspace{-0.8ex}\delta\hspace{-0.8ex}>)$. 

\begin{proof}
By linearity of the intersection form it is sufficient to prove that if $S$ is a complete intersection cut out by
general elements $D_j \in |m_j H_j|$ for $m_j\gg0$, then
\[
D_1 \cdots D_{n-2} \cdot c_2(E\hspace{-0.8ex}<\hspace{-0.8ex}\delta\hspace{-0.8ex}>) 
= c_2(E\hspace{-0.8ex}<\hspace{-0.8ex}\delta\hspace{-0.8ex}>\hspace{-0.8ex}|_S)  \geq 0.
\]
Since $X$ is smooth in codimension two, the surface $S$ is smooth. The reflexive sheaf $E$ being locally free in codimension two,
the restriction $E|_S$ is a vector bundle. Moreover 
$E\hspace{-0.8ex}<\hspace{-0.8ex}\delta\hspace{-0.8ex}>\hspace{-0.8ex}|_S$ is generically nef 
and $c_1( E\hspace{-0.8ex}<\hspace{-0.8ex}\delta\hspace{-0.8ex}>\hspace{-0.8ex}|_S)$ is nef. 
We conclude with \cite[Thm. 8']{LM97}.
\end{proof}

\begin{corollary}\footnote{This statement seems now a bit too optimistic to me, see 
\url{http://math.unice.fr/\~hoering/articles/remark-bs-conjecture.pdf} for a correction.} 
 \label{corollarymiyaokaQ}
Let $X$ be a normal, projective variety of dimension $n \geq 2$ that is smooth in codimension two. 
Let $D$ be a nef $\Q$-Cartier $\Q$-divisor on $X$
such that $\Omega_X\hspace{-0.8ex}<\hspace{-0.8ex}\frac{1}{n}D\hspace{-0.8ex}>$ is generically nef and $K_X+D$ is nef in codimension one. 
Then we have
\[
H_1 \cdots H_{n-2}  \cdot c_2(X) 
\geq - H_1 \cdots H_{n-2}   \cdot (\frac{n-1}{n} K_{X} \cdot D + \frac{n-1}{2n} D^2),
\]
where $H_1, \ldots, H_{n-2}$ is any collection of nef Cartier divisors on $X$. 
\end{corollary}

We recall that the usual formulas for tensor products of vector bundles extends to $\Q$-vector bundles \cite[Ch.6.2, Ch.8.1]{Laz04b}:
let $X$ be a normal, projective variety that is smooth in codimension two, 
and let $E$ be a coherent sheaf of rank $r$ over $X$.
If $\delta \in N^1(X)_\Q$ is a numerical class, then
\begin{equation} \label{qc1}
c_1(E\hspace{-0.8ex}<\hspace{-0.8ex}\delta\hspace{-0.8ex}>) = c_1(E) + r \delta
\end{equation}
\begin{equation} \label{qc2}
c_2(E\hspace{-0.8ex}<\hspace{-0.8ex}\delta\hspace{-0.8ex}>) = c_2(E) + (r-1) c_1(E) \cdot \delta + \frac{r (r-1)}{2} \delta^2.
\end{equation}

\begin{proof}[Proof of Corollary \ref{corollarymiyaokaQ}]
By  the linearity of the intersection form, 
it is sufficient to show the statement in the case where the Cartier divisors $H_i$ are ample.
Moreover by Lemma \ref{lemmagenericallynef},a) the sheaf $\Omega_X\hspace{-0.8ex}<\hspace{-0.8ex}\frac{1}{n}D\hspace{-0.8ex}>$ is generically nef if and only its bidual $\Omega_X^{**}\hspace{-0.8ex}<\hspace{-0.8ex}\frac{1}{n}D\hspace{-0.8ex}>$ is generically nef.
Since $X$ is smooth in codimension two, we have $c_2(\Omega_X)=c_2(\Omega_X^{**})$.
Therefore Lemma \ref{lemmamiyaokaQ} applies and yields
$$
H_1 \cdots H_{n-2} \cdot c_2(\Omega_X\hspace{-0.8ex}<\hspace{-0.8ex}\frac{1}{n}D\hspace{-0.8ex}>) \geq 0.
$$ 
Since by Formula \eqref{qc2}
$$
c_2(\Omega_X\hspace{-0.8ex}<\hspace{-0.8ex}\frac{1}{n}D\hspace{-0.8ex}>) = c_2(X) + \frac{n-1}{n} K_X \cdot D + \frac{n-1}{2n} D^2,
$$
we get
$$
H_1 \cdots H_{n-2} \cdot c_2(X) \geq - H_1 \cdots H_{n-2} \cdot (\frac{n-1}{n} K_X \cdot D + \frac{n-1}{2n} D^2).
$$
\end{proof}

\subsection{Some technical lemmas}

The following lemma shows that we can reduce the non-vanishing problem to non-singular varieties.

\begin{lemma} \label{lemmareduction}
Let $X$ be a normal, projective variety of dimension $n$, and
let $A$ be a Cartier divisor on $X$. Let \holom{\nu}{X'}{X}
be a desingularisation. Then for all $j \in \Z$ we have an inclusion:
$$
H^0(X',  \sO_{X'}(K_{X'}+j\nu^*A))
\subseteq
H^0(X,  \sO_X(K_X+jA)).
$$ 
\end{lemma}

\begin{proof}
Since $A$ is Cartier, the projection formula yields
$$
\nu_* \sO_{X'}(K_{X'}+j\nu^*A) \simeq \nu_* \sO_{X'}(K_{X'}) \otimes \sO_X(jA).
$$
Note that since  $\nu_* \sO_{X'}(K_{X'}) \otimes \sO_X(jA)$ is torsion-free we have an inclusion
\[
\nu_* \sO_{X'}(K_{X'}) \otimes \sO_X(jA) \hookrightarrow (\nu_* \sO_{X'}(K_{X'}) \otimes \sO_X(jA))^{**}.
\] 
Moreover for any reflexive sheaf $\sF$ on a normal variety we have
\[
j_* (\sF|_{X_{\nons}}) \simeq \sF
\]
where $j: X_{\nons} \hookrightarrow X$ is the inclusion.
Thus $(\nu_* \sO_{X'}(K_{X'}) \otimes \sO_X(jA))^{**}$ and $\sO_X(K_X+jA)$ are isomorphic since
they coincide on $X_{\nons}$ and we get an inclusion
$$
\nu_* \sO_{X'}(K_{X'}+j\nu^*A) \hookrightarrow \sO_X(K_X+jA).
$$
\end{proof}

\begin{proposition} \label{propositionbasic}
Let $X$ be a normal, projective variety of dimension $n$, and
let $A$ be a nef and big Cartier divisor on $X$. 
Then the following holds:
\begin{enumerate}
\item There exists a $j \in \{ 1, \ldots, n+1 \}$ such that 
$H^0(X,  \sO_X(K_X+jA)) \neq 0$. In particular the divisor $K_X+(n+1)A$ is generically nef.
\item If $(K_X+nA) \cdot A^{n-1} \geq 0$,
there exists a $j \in \{ 1, \ldots, n \}$ such that $H^0(X, \sO_X(K_X+jA)) \neq 0$.
\item If $(K_X+nA) \cdot A^{n-1} < 0$,
there exists a birational morphism \holom{\tau}{X}{\PP^n} such that $\sO_X(A) \simeq \tau^* \sO_{\PP^n}(1)$.  
\end{enumerate}
\end{proposition}

\begin{proof}
By Lemma \ref{lemmareduction} statement a) follows from \cite[Prop.9.4.23]{Laz04b}.

b)+c) Let \holom{\nu}{X'}{X} be a desingularisation. We have $\nu_* K_{X'} = K_X$, so by the
projection formula 
\[
(K_X+nA) \cdot A^{n-1} = (K_{X'}+n\nu^*A) \cdot \nu^* A^{n-1}
\]
Thus the condition lifts to $X'$ and we conclude by Lemma \ref{lemmareduction} and \cite[Thm.2.2]{Fuj89}.
\end{proof}

The following basic fact is well-known to experts. For the convenience of the reader we include
a proof.

\begin{lemma} \label{lemmadirectimage}
Let \holom{\varphi}{X}{Y} be a fibration between projective manifolds $X$ and $Y$, and let $A$
nef and big $\Q$-Cartier $\Q$-divisor on $X$.
Suppose that for a general fibre $F$ one has
\[
H^0(F, \sO_F(D)) \neq 0,
\]
where $D$ is a Cartier divisor on $F$ such that $D \sim_\Q K_F+A|_F$.
Then $K_{X/Y}+A$ is pseudoeffective.
\end{lemma}

\begin{proof}
It is sufficient to show that $m(K_{X/Y}+A)$ is pseudoeffective for $m \gg 0$ sufficiently divisible
so we choose $m \in \N$ such that $mA$ is Cartier and $H^0(F, \sO_F(m K_F+mA|_F)) \neq 0$.
Hence the direct image sheaf $\varphi_* (\sO_X(m(K_{X/Y}+A)))$ is not zero.
The Cartier divisor $mA$ is nef and big, 
so it follows from \cite[Ch.2]{Vie95}, \cite[Thm.4.13]{Cam04b}, \cite[Thm.0.2]{BP08} that
$\varphi_* (\sO_X(m(K_{X/Y}+A)))$ is weakly positive in the sense of Viehweg.
Since $\sO_X(m(K_{X/Y}+A))$ has rank one, the canonical morphism
\[
\varphi^* \varphi_* (\sO_X(m(K_{X/Y}+A))) \rightarrow \sO_X(m(K_{X/Y}+A))
\]
is generically surjective, so $\sO_X(m(K_{X/Y}+A))$ is also weakly positive.
Thus the
divisor $m(K_{X/Y}+A)$ is pseudoeffective. 
\end{proof}

If $A$ is a Cartier divisor we can combine 
Lemma \ref{lemmadirectimage} and Proposition \ref{propositionbasic} to obtain:

\begin{proposition} \label{propositiondirectimage}
Let $X$ be a projective manifold of dimension $n$.
Let  \holom{\mu}{X'}{X} and  \holom{\varphi}{X'}{Y} be a model of the MRC-fibration 
(cf. Theorem \ref{theoremMRC}), and denote by $m$ the dimension of $Y$. 
Let $A$ be a nef and big Cartier divisor on $X$.
Then 
$$
K_{X'/Y}+(n-m+1)\mu^* A
$$ 
is pseudoeffective. 
If 
$$
K_{X'/Y}+(n-m) \mu^* A
$$ 
is not pseudoeffective, the general $\varphi$-fibre $F$ admits 
a birational morphism \holom{\tau}{F}{\PP^{n-m}} such that  $\sO_F(A) \simeq \tau^* \sO_{\PP^{n-m}}(1)$.

In particular $K_{X}+(n-m+1)A$ is pseudoeffective.  If 
$K_{X}+(n-m)A$ is not pseudoeffective, the manifold $F$ admits 
a birational morphism \holom{\tau}{F}{\PP^{n-m}} such that $\sO_F(A) \simeq \tau^* \sO_{\PP^{n-m}}(1)$.
\end{proposition}

The second statement of the proposition is a consequence of the first and the 
fundamental result due to Boucksom, Demailly, P{\u a}un and Peternell \cite[Cor.0.3]{BDPP04} on the
pseudoeffectiveness of the canonical bundle of a non-uniruled, projective manifold.

\section{The cotangent sheaf of uniruled varieties}

Theorem \ref{theoremgenericnefCartier} will be a consequence of the following statement.

\begin{theorem}\label{theoremgenericnef}
Let $X$ be a normal, projective variety of dimension $n$.
Let $A$ be a nef and big $\Q$-Cartier $\Q$-divisor on $X$.
Then the $\Q$-twisted sheaf $\Omega_X\hspace{-0.8ex}<\hspace{-0.8ex}A\hspace{-0.8ex}>$ is generically nef (cf. Definition \ref{definitiongenericallynef}) unless there exists
a birational morphism \holom{\mu}{X'}{X} from a projective manifold $X'$
and a fibration \holom{\varphi}{X'}{Y} onto a projective manifold $Y$ of dimension $m<n$
such that the general fibre $F$ is rationally connected and 
\[
H^0(F, \sO_F(D)) = 0
\]
where $D$ is any Cartier divisor on $F$ such that $D \sim_\Q K_F+j \mu^*A|_F$ with $j \in [0, n-m] \cap \Q$.
\end{theorem}

Let \holom{\varphi}{X}{Y} be a fibration between projective manifolds, and
let $\Omega_X \rightarrow \Omega_{X/Y} \rightarrow 0$ be the canonical map between the sheaves of K\"ahler
differentials. We define the relative tangent sheaf $T_{X/Y}$ to be the saturation of
\[
\Omega_{X/Y}^* \rightarrow \Omega_X^* =: T_X
\] 
in $T_X$, and $\det T_{X/Y}$ the divisor corresponding to its determinant.  
The main difficulty of the proof is that in general
the relative canonical bundle of a fibration does not coincide with the dual of $\det T_{X/Y}$.
We overcome this difficulty by making an appropriate base change.

\begin{proof}[Proof of Theorem \ref{theoremgenericnef}] 
Let us assume that $\Omega_X\hspace{-0.8ex}<\hspace{-0.8ex}A\hspace{-0.8ex}>$ is not generically nef.
We fix $L_1, \ldots, L_{n-1}$ ample Cartier divisors on $X$ such that $\Omega_X\hspace{-0.8ex}<\hspace{-0.8ex}A\hspace{-0.8ex}>$ is not generically nef
with respect to $L_1, \ldots, L_{n-1}$. Let
$$
C=D_1 \cap \ldots \cap D_{n-1}
$$ 
be a MR-general curve
where $D_i \in |m_i L_i|$ general and $m_i\gg0$ such that $\Omega_X\hspace{-0.8ex}<\hspace{-0.8ex}A\hspace{-0.8ex}>\hspace{-0.8ex}|_C$ is not nef.
If $\sF\hspace{-0.8ex}<\hspace{-0.8ex}A\hspace{-0.8ex}>$ is a non-zero torsion-free $\Q$-twisted sheaf on $X$, 
we define the slope
$$
\mu(\sF\hspace{-0.8ex}<\hspace{-0.8ex}A\hspace{-0.8ex}>) := \frac{c_1(\sF\hspace{-0.8ex}<\hspace{-0.8ex}A\hspace{-0.8ex}>\hspace{-0.8ex}|_C)}{\rk \sF}.
$$ 
By Equation \eqref{qc1} one has
\[
\frac{c_1(\sF\hspace{-0.8ex}<\hspace{-0.8ex}A\hspace{-0.8ex}>\hspace{-0.8ex}|_C)}{\rk \sF}= \frac{c_1(\sF|_C)}{\rk \sF} + A \cdot  C.
\]
By definition the $\Q$-twisted sheaf $\sF\hspace{-0.8ex}<\hspace{-0.8ex}A\hspace{-0.8ex}>$ is semistable if for every non-zero torsion-free subsheaf $\sE \subset \sF$, we have $\mu(\sE\hspace{-0.8ex}<\hspace{-0.8ex}A\hspace{-0.8ex}>) \leq \mu(\sF\hspace{-0.8ex}<\hspace{-0.8ex}A\hspace{-0.8ex}>)$.

Denote by $T_X:=\Omega_X^*$ the tangent sheaf of $X$,
and let 
\[
0= \sF_0 \subsetneq \sF_1 \subsetneq \ldots \subsetneq \sF_r=T_X
\]
be the Harder-Narasimhan filtration of $T_X$ with respect to $L_1, \ldots, L_{n-1}$. 
Then for $i=1, \ldots, r$, the graded pieces $\sG_i:=\sF_i/\sF_{i-1}$ are semistable torsion-free sheaves
and if $\mu(\sG_i)$ denotes the slope, we have a strictly decreasing sequence
\[
\mu(\sG_1) > \mu(\sG_2) > \ldots > \mu(\sG_r).
\]
Since twisting with a $\Q$-Cartier $\Q$-divisor does not change the stability properties of a torsion-free
sheaf, the Harder-Narasimhan filtration of $T_X\hspace{-0.8ex}<\hspace{-0.8ex}-A\hspace{-0.8ex}>$ is 
\[
0= \sF_0\hspace{-0.8ex}<\hspace{-0.8ex}-A\hspace{-0.8ex}> \subsetneq \sF_1\hspace{-0.8ex}<\hspace{-0.8ex}-A\hspace{-0.8ex}> \subsetneq \ldots \subsetneq \sF_r\hspace{-0.8ex}<\hspace{-0.8ex}-A\hspace{-0.8ex}>=T_X\hspace{-0.8ex}<\hspace{-0.8ex}-A\hspace{-0.8ex}>
\]
with graded pieces $\sG_i\hspace{-0.8ex}<\hspace{-0.8ex}-A\hspace{-0.8ex}>$ and slopes
\[
\mu(\sG_i\hspace{-0.8ex}<\hspace{-0.8ex}-A\hspace{-0.8ex}>) = \mu(\sG_i)-A \cdot C.
\]
We claim that
\begin{equation} \label{equationstar}
\mu(\sG_1\hspace{-0.8ex}<\hspace{-0.8ex}-A\hspace{-0.8ex}>)=\mu(\sF_1\hspace{-0.8ex}<\hspace{-0.8ex}-A\hspace{-0.8ex}>) > 0.
\end{equation}
Otherwise the slopes of all the graded pieces $\sG_i\hspace{-0.8ex}<\hspace{-0.8ex}-A\hspace{-0.8ex}>$ are non-positive.
By the Mehta-Ramanathan theorem \cite[Thm.6.1]{MR82} the Harder-Narasimhan filtration commutes with restriction to $C$, 
so the $\Q$-twisted vector bundles $\sG_i\hspace{-0.8ex}<\hspace{-0.8ex}-A\hspace{-0.8ex}>\hspace{-0.8ex}|_C$ are semistable of non-positive slope, hence 
antinef. Thus $\Omega_X\hspace{-0.8ex}<\hspace{-0.8ex}A\hspace{-0.8ex}>\hspace{-0.8ex}|_C$ is an extension of nef $\Q$-vector bundles, hence nef.
This contradicts our hypothesis.

The $\Q$-Cartier divisor $A$ being nef $\mu(\sF_1\hspace{-0.8ex}<\hspace{-0.8ex}-A\hspace{-0.8ex}>) > 0$ implies $\mu(\sF_1)>0$, 
so $\sF_1|_C$ is ample. We know by
standard arguments in stability theory \cite[p.61ff]{MP97} 
that $\sF_1$ is integrable, moreover the MR-general curve $C$ does not meet the singular
locus of the foliation by Remark \ref{remarkmrgeneral}.
Thus we can apply the
Bogomolov-McQuillan theorem \cite[Thm.0.1]{BM01}, \cite[Thm.1]{KST07} 
to see that the closure of a $\sF_1$-leaf
through a generic point of $C$ is algebraic and rationally connected. 
Since $C$ moves in a covering family the  generic $\sF_1$-leaves are algebraic with rationally connected closure.
If $\chow{X}$ denotes the Chow scheme of $X$, we get a rational map $X \dashrightarrow \chow{X}$ that sends a 
general point $x$ to the closure of the unique leaf through $x$. 
Let $Y$ be a desingularisation of the closure of the image,
and let $X'$ be a desingularisation of the universal family over $Y$. 
By construction the natural map \holom{\mu}{X'}{X} is birational and the general fibres of the fibration  $\holom{\varphi}{X'}{Y}$
map onto the closure of general $\sF_1$-leaves.

By Remark \ref{remarkmrgeneral} the MR-general curve $C$ does not meet the exceptional locus of $\mu$,
so we can see it as a curve in $X'$.
Denote by $X_C$ the normalisation of the fibre product $X' \times_Y C \subset X' \times C$,
and let \holom{p_X}{X_C}{X} the projection on the first factor.
The fibration $X' \times_Y C \rightarrow C$
admits a natural section 
$$
C \rightarrow X' \times_Y C \subset X' \times C, \ c \ \mapsto (c,c), 
$$ 
by the universal property of the normalisation we get a section of \holom{p_C}{X_C}{C}
which we denote by \holom{s}{C}{X_C}.
By \cite[Rem.19]{KST07} the normal variety $X_C$ is smooth in an analytic neighbourhood $U \subset X_C$ of $s(C)$ 
and 
\[
T_{X_C/C}|_U \simeq (p_X^* \mu^* \sF_1)|_U. 
\]
In particular by the inequality \eqref{equationstar}, one has 
$$
(\det T_{X_C/C} - (n-m) p_X^* \mu^* A) \cdot s(C) = (\det \sF_1-(n-m)A) \cdot C = (n-m) \mu(\sF_1\hspace{-0.8ex}<\hspace{-0.8ex}-A\hspace{-0.8ex}>) > 0.
$$
Since $s(C)$ is a section of the fibration it does not meet any multiple fibre components, so 
$-K_{X_C/C}$ and $\det T_{X_C/C}$ coincide in a neighbourhood of $s(C)$.
Thus
\begin{equation} \label{equationdoublestar}
\qquad (-K_{X_C/C} - (n-m) p_X^* \mu^* A) \cdot s(C)
=
(\det T_{X_C/C} - (n-m) p_X^* \mu^* A) \cdot s(C)
> 0.
\end{equation}
Since $X_C$ is smooth in a neighbourhood of $s(C)$, we can replace $X_C$ by a desingularisation
without changing the inequality \eqref{equationdoublestar}.
We will now argue by contradiction and suppose that there exists a
Cartier divisor $D$ on a general $\varphi$-fibre $F$ such that $D \sim_\Q K_{F}+j \mu^*A$ for some $j \in [0, n-m] \cap \Q$ and
\[
H^0(F, \sO_F(D)) \neq 0.
\]
Since the general $p_C$-fibre is a general $\varphi$-fibre this implies by 
Lemma \ref{lemmadirectimage} that
$K_{X_C/C}+j p_X^* \mu^* A$ is pseudoeffective.
Since $s(C)$ is a section, its normal bundle is isomorphic to $T_{X_C/C}|_{s(C)} \simeq \sF_1|_{C}$
which is ample. This implies by \cite[Cor.8.4.3]{Laz04b}
that $E \cdot s(C) \geq 0$ for every effective divisor $E \subset X_C$, hence
\[
(K_{X_C/C}+ (n-m) p_X^* \mu^* A) \cdot s(C) \geq (K_{X_C/C}+ j p_X^* \mu^* A) \cdot s(C) \geq 0.
\]
This contradicts the inequality \eqref{equationdoublestar}.
\end{proof}

\begin{proof}[Proof of Theorem \ref{theoremgenericnefCartier}]
Suppose that $\Omega_X\hspace{-0.8ex}<\hspace{-0.8ex}A\hspace{-0.8ex}>$ is not generically nef. Applying Theorem \ref{theoremgenericnef}
yields a birational morphism \holom{\mu}{X'}{X} from a projective manifold $X'$
and a fibration \holom{\varphi}{X'}{Y} onto a projective manifold $Y$ of dimension $m$
such that the general fibre $F$ satisfies
\[
H^0(F, \sO_F(K_F+ j \mu^* A)) = 0 \qquad \forall \ j \in \{ 1, \ldots, n-m \}.
\]
It follows from Prop.\ref{propositionbasic}.b) and Prop.\ref{propositionbasic}.c) that $(X,A)$ is birationally a scroll.
\end{proof}

In Section \ref{sectionBS} we use the following technical lemma:

\begin{lemma} \label{lemmatechnical}
In the situation of the proof of Theorem \ref{theoremgenericnef}, suppose that $A$ is a Cartier divisor.
Suppose moreover that $\Omega_X\hspace{-0.8ex}<\hspace{-0.8ex}-A\hspace{-0.8ex}>$ is not generically nef,
i.e. 
\[
\mu(\sG_1\hspace{-0.8ex}<\hspace{-0.8ex}-A\hspace{-0.8ex}>) > 0.
\]
Denote by $l \in \N$ the maximal number such that 
\[
\mu(\sG_i\hspace{-0.8ex}<\hspace{-0.8ex}-A\hspace{-0.8ex}>) > 0 \qquad \forall \ i \in \{ 1, \ldots, l \}.
\]
Then the following holds:

a) For every $i \in \{ 1, \ldots, l \}$ we have 
\[
\mu(\sF_i\hspace{-0.8ex}<\hspace{-0.8ex}-\frac{\rk \sF_i+1}{\rk \sF_i}A\hspace{-0.8ex}>) \leq 0.
\]

b) There exists a sequence of rational numbers $w_1, \ldots, w_l$ such that
\[
w_i \in [ \rk \sG_i, \rk \sG_i+1 ]  \qquad \forall \ i \in \{ 1, \ldots, l \}
\]
and
\[
\sum_{i=1}^l w_i = (\sum_{i=1}^l \rk \sG_i)+1
\]
and
\[
\mu(\sG_i\hspace{-0.8ex}<\hspace{-0.8ex}-\frac{w_i}{\rk \sG_i}A\hspace{-0.8ex}>) \leq 0 \qquad \forall \ i \in \{ 1, \ldots, l \}.
\]
\end{lemma}

\begin{proof} 
For statement a) we argue as in the proof of Theorem \ref{theoremgenericnef}: For $i \in \{ 1, \ldots, l \}$
the saturated subsheaf $\sF_i \subset T_X$ is integrable and
we get a birational morphism \holom{\mu_i}{X_i}{X} and a fibration $\holom{\varphi_i}{X_i}{Y_i}$ 
such that $\sF_i$ corresponds to the relative tangent sheaf of $\varphi_i$.
We argue by contradiction and suppose that
\[
\mu(\sF_i\hspace{-0.8ex}<\hspace{-0.8ex}-\frac{\rk \sF_i+1}{\rk \sF_i}A\hspace{-0.8ex}>) > 0.
\]
As in the proof of  Theorem \ref{theoremgenericnef} we see 
that the general $\varphi_i$-fibre $F_i$ satisfies
\[
H^0(F_i, \sO_F(K_{F_i}+ j \mu^* A)) = 0 \qquad \forall j \in \{ 1, \ldots, \rk \sF_i+1 \}. 
\]
This contradicts Proposition \ref{propositionbasic}, a).

If $l=1$ the statement b) is an immediate consequence of a): just take $w_1=\rk \sG_1+1$.
Suppose now that $l>1$ and
denote by $C$ the MR-general curve we used in the proof of Theorem \ref{theoremgenericnef} to compute the slopes.
By statement a) we have
\[
[c_1(\sF_i)-(\rk \sF_i+1)A] \cdot C \leq 0 \qquad \forall \ i \in \{ 1, \ldots, l \}.
\]
Since the sheaves $\sG_d$ are the graded pieces of the filtration $\sF_\bullet$ this 
implies that for all $i \in \{ 1, \ldots, l \}$ we have
\[
(*_i) \qquad [\sum_{d=1}^i c_1(\sG_d)-((\sum_{d=1}^i \rk \sG_d)+1)A] \cdot C \leq 0.
\]
We will now construct a sequence $w_i$ inductively by using the inequalities $(*_i)$. 

{\em Start of the induction $i=1$.}
By hypothesis have
\[
(c_1(\sG_1)-\rk \sG_1 A) \cdot C > 0
\]
and since $\sG_1=\sF_1$ by $(*_1)$
\[
[c_1(\sG_1)-(\rk \sG_1+1)A] \cdot C \leq 0.
\]
We define $w_1$ to be the unique rational number such that
\[
(c_1(\sG_1)-w_1 A) \cdot C = 0,
\]
i.e. the slope of $\sG_1\hspace{-0.8ex}<\hspace{-0.8ex}-\frac{w_1}{\rk \sG_1}A\hspace{-0.8ex}>$
equals zero.

{\em Induction step $i-1 \rightarrow i$.}
We have constructed so far $w_1, \ldots, w_{i-1}$ such that
\[
w_d \in [ \rk \sG_d, \rk \sG_d+1 ]  \qquad \forall \ d \in \{ 1, \ldots, i-1 \}
\]
and\footnote{For $i=2$ the inequality $(**_2)$ is just $w_1 \in [ \rk \sG_1, \rk \sG_1+1 ]$, for
$i>2$ this will be established at the end of the preceding induction step.}
\[
(**_i) \qquad \sum_{d=1}^{i-1} w_d \leq (\sum_{d=1}^{i-1} \rk \sG_d)+1
\]
and
\[
(c_1(\sG_d)-w_d A) \cdot C = 0 \qquad \forall \ d \in \{ 1, \ldots, i-1 \}.
\]
Plugging these equalities in the inequality $(*_i)$ we obtain
\[
\left[
c_1(\sG_i)-
\left(
\rk \sG_i + (\sum_{d=1}^{i-1} \rk \sG_d)+1-(\sum_{i=d}^{i-1} w_d)
\right)
A
\right] 
\cdot C \leq 0.
\]
By $(**_i)$ we have $(\sum_{d=1}^{i-1} \rk \sG_d)+1-(\sum_{d=1}^{i-1} w_d) \geq 0$
and by hypothesis
\[
(c_1(\sG_i)-\rk \sG_i A) \cdot C > 0
\]
If $i<l$ we define $w_i$ to be the unique rational number such that
\[
(c_1(\sG_i)-w_i A) \cdot C = 0,
\]
i.e. the slope of $\sG_i\hspace{-0.8ex}<\hspace{-0.8ex}-\frac{w_i}{\rk \sG_i}A\hspace{-0.8ex}>$
equals zero.
Since we have
\[
w_i \in [ \rk \sG_i, \rk \sG_i + (\sum_{d=1}^{i-1} \rk \sG_d)+1-(\sum_{d=1}^{i-1} w_d) ]
\]
we see immediately that $(**_{i+1})$ holds, so the induction can continue.

If $i=l$ we set 
\[
w_l := \rk \sG_l + (\sum_{d=1}^{l-1} \rk \sG_d)+1-(\sum_{d=1}^{l-1} w_d),
\]
so we have $\sum_{i=1}^l w_i = (\sum_{i=1}^l \rk \sG_i)+1$.
\end{proof}

\section{The Beltrametti-Sommese conjecture} 
\label{sectionBS}

The following lemma is the technical cornerstone of our approach.

\begin{lemma} \label{lemmaminusone}
Let $X$ be a projective manifold, and let $A$ be a Cartier divisor on $X$. Suppose that
$1, \ldots, n-1$ are roots of the Hilbert polynomial $\chi(X, \sO_X(K_X+tA))$. Then one has
\begin{equation} \label{equationone}
\chi(X, \sO_X) + \frac{1}{2}  A^{n-1} \cdot (K_X+(n-1)A)=0
\end{equation}
and
\begin{equation} \label{equationtwo}
A^{n-2} \cdot
[2  (K_X^2+c_2(X))
+
6n A \cdot K_X 
+
(n+1) (3n-2) A^2
]=0.
\end{equation}
\end{lemma}

\begin{remark*}
For $n=2$ the left hand side of Equation \eqref{equationone} and \eqref{equationtwo} 
are (multiples of) the Riemann-Roch formula for $\chi(X, \sO_X(K_X+A))$.
This corresponds well with the origin of the Beltrametti-Sommese conjecture \cite[Ch. 7.2]{BS95}:
the linear system $K_X+(n-1)A$ should behave as an adjoint linear system on a surface.  
\end{remark*}

\begin{proof}
By hypothesis 
$$
\chi(X, \sO_X(K_X+tA))=\frac{A^n}{n!} (t-a)   \prod_{j=1}^{n-1} (t-j),
$$
where $a$ is a parameter.
Since 
$$
\prod_{j=1}^{n-1} (t-j)
=
t^{n-1} - (\sum_{j=1}^{n-1} j) t^{n-2} 
+ (\sum_{\stackrel{j,k=1}{j<k}}^{n-1} j k) t^{n-3} - \ldots + (-1)^{n-1} (n-1)!,
$$
we have
$$
(t-a) \prod_{j=1}^{n-1} (t-j)
=
t^n - (a+\sum_{j=1}^{n-1} j) t^{n-1} 
+ (\sum_{\stackrel{j,k=1}{j<k}}^{n-1} j k+ a \sum_{j=1}^{n-1} j) t^{n-2}
- \ldots +
(-1)^{n} a (n-1)!.
$$
Comparing coefficients with Riemann-Roch formula \eqref{RR}, we get
\begin{eqnarray*}
\frac{A^{n-1} \cdot  K_X}{2(n-1)!} 
&=& -  
(a+\sum_{j=1}^{n-1} j) \frac{A^n}{n!},
\\
\frac{A^{n-2} \cdot (K_X^2+c_2(X))}{12(n-2)!}
&=&
 (\sum_{\stackrel{j,k=1}{j<k}}^{n-1} j k+ a \sum_{j=1}^{n-1} j) \frac{A^n}{n!},
\\
\chi(X, \sO_X) &=&  a \frac{A^n}{n}
\end{eqnarray*}
The statement follows by plugging these expressions into the Equations \eqref{equationone} and \eqref{equationtwo} and using the
elementary formula
\[
\sum_{\stackrel{j,k=1}{j<k}}^{n-1} j k = \frac{1}{24} (n-2) (n-1) n (3n-1).
\]
\end{proof}

\begin{proof}[Proof of Theorem \ref{theoremBS}]
We argue by contradiction and suppose that $H^0(X, \sO_X(K_X+jA))=0$ for all $j \in \{1, \ldots, n-1\}$.
Let \holom{\nu}{X'}{X} be a resolution of singularities, then by
Lemma \ref{lemmareduction} one has
\[
H^0(X', \sO_{X'}(K_{X'}+j\nu^*A))=0 \qquad \forall \ j \in \{1, \ldots, n-1\}.
\]
Since $\nu^* A$ is nef and big, the Kawamata-Viehweg theorem implies that
\[
\chi(X', \sO_{X'}(K_{X'}+j\nu^*A)) = h^0(X', \sO_{X'}(K_{X'}+j\nu^*A))=0 \qquad \forall \ j \in \{ 1, \ldots, n-1\},
\]
in particular Lemma \ref{lemmaminusone} applies.
Let $Y$ be the base of the MRC-fibration of $X'$. 

\begin{center}
{\bf  Case I: $\dim Y=0$.}
\end{center}

Since $A$ is nef and $K_X+(n-1)A$ is generically nef, one has 
\begin{equation} \label{inequalitya}
(K_{X'}+(n-1)\nu^* A) \cdot (\nu^* A)^{n-1} = (K_{X}+(n-1)A) \cdot A^{n-1} \geq 0.
\end{equation}
By Lemma \ref{lemmaminusone} we have
\[
\chi(X', \sO_{X'}) + \frac{1}{2} (K_{X'}+(n-1)\nu^*A) \cdot (\nu^* A)^{n-1}=0.
\]
Yet $X'$ is rationally connected, so $\chi(X', \sO_{X'})=1$. 
This contradicts the inequality \eqref{inequalitya}.

\begin{center}
{\bf  Case II: $\dim Y=1$.}
\end{center}

Since the base of the MRC-fibration has dimension one, 
we have a morphism \holom{\varphi}{X'}{Y} 
onto a smooth curve of genus at least one (cf. Remark \ref{remarksRCquotient}). 
By Proposition \ref{propositiondirectimage} the divisor $K_{X'}+(n-1)\nu^*A$ is generically nef
unless $(X', \nu^* A)$ (and hence $(X,A)$) is birationally a scroll with base $Y$. 

a) If $(X,A)$ is not birationally a scroll with base $Y$, denote by $F'$ a general $\varphi$-fibre, then 
by Proposition \ref{propositionbasic} there exists a $j \in \{1, \ldots, n-1 \}$
such that $H^0(F', \sO_{F'}(K_{F'}+j\nu^*A)) \neq 0$. In particular the direct image sheaf
$ \varphi_* \sO_{X'}(K_{X'/Y}+j\nu^*A)$ 
is not zero and an ample vector bundle by \cite[Cor.3.7]{Vie01}.
Thus
\begin{eqnarray*}
h^0(X', \sO_{X'}(K_{X'}+j\nu^*A)) &=& h^0(Y,  \sO_Y(K_Y) \otimes  \varphi_* \sO_{X'}(K_{X'/Y}+j\nu^*A)) 
\\
& \geq & \chi(Y,   \sO_Y(K_Y) \otimes  \varphi_* \sO_{X'}(K_{X'/Y}+j\nu^*A)) > 0
\end{eqnarray*}
by an easy Riemann-Roch computation for vector bundles on curves.

b) If $(X,A)$ is birationally a scroll with base $Y$,  then by assumption $X$ has rational singularities (cf. the statement of 
Theorem \ref{theoremBS}). The Albanese morphism $\holom{\alpha}{X'}{Alb(X')}$ identifies to the composition of 
the MRC-fibration \holom{\varphi}{X'}{Y} and the embedding \holom{\alpha_Y}{Y}{Alb(Y)}.
Since $X$ has rational singularities the Albanese map of $X'$ factors through $\nu$ \cite[Lemma 2.4.1]{BS95},
so we get a fibration $\holom{\psi}{X}{Y}$ such that $\varphi = \psi \circ \nu$ . 
A general $\psi$-fibre $F$ is a Cartier divisor in $X$, so
\[
(K_F+(n-1)A|_F) \cdot A|_F^{n-2} = (K_X+(n-1)A) \cdot F \cdot A^{n-2} \geq 0.
\]
In particular by Proposition \ref{propositionbasic}.b) there exists a $j \in \{1, \ldots, n-1 \}$
such that $H^0(F, \sO_F(K_F+jA)) \neq 0$. Since $\nu_* \sO_{X'}(K_{X'}) \simeq \sO_X(K_X)$ this shows that
the direct image sheaf
$$ 
\varphi_* \sO_{X'}(K_{X'/Y}+j\nu^*A) \simeq \psi_* \sO_X(K_{X/Y}+jA)
$$ 
is not zero and we conclude as in a).

The following example shows why our strategy of proof does not apply if $(X,A)$ birationally a scroll
and $X$ has irrational singularities.

\begin{example}
Let $C \subset \PP^2$ be a smooth curve of degree three, and set $\sO_C(1)$ for
the restriction of the hyperplane divisor to $C$.
Denote by $A'$ the tautological divisor on the projectivised bundle $\holom{\varphi}{S':=\PP(\sO_C \oplus \sO_C(1))}{C}$.
Then $\sO_{S'}(A')$ is globally generated and induces a birational map $\holom{\nu}{S'}{S \subset \PP^3}$
that contracts the section corresponding to the quotient bundle 
$\sO_C \oplus \sO_C(1) \rightarrow \sO_C$. The surface $S$ has degree three and is of course the cone
over the elliptic curve $C$. Thus $S$ is normal, Gorenstein and $K_S = - H|_S$, where $H$ is the hyperplane divisor. 
The Cartier divisor $A:=H$ is ample, the adjoint bundle $\sO_S(K_S+A)$ is trivial, so nef and 
\[
H^0(S, \sO_S(K_S+A)) = \C.
\] 
It is not possible to prove the existence of this global section by looking only at the nonsingular surface $S'$:
the divisor $K_{S'}+\nu^* A = K_{S'}+A'$ is not generically nef, its restriction to a $\varphi$-fibre is $\sO_{\PP^1}(-1)$. 
\end{example}

\begin{center}
{\bf  Case III: $\dim Y=2$.}
\end{center}

In order to simplify the notation we assume without loss of generality that $X$ is smooth, so $X'=X$.
Note that by Proposition \ref{propositiondirectimage} the divisor $K_X+(n-1)A$ is pseudoeffective.
By Lemma \ref{lemmaminusone} one has
\[
\chi(X, \sO_{X}) + \frac{1}{2} (K_{X}+(n-1)A) \cdot A^{n-1}=0.
\]
Since $K_X+(n-1)A$ is pseudoeffective we get a contradiction if 
$\chi(X, \sO_X)>0$. 

Suppose now that $\chi(X, \sO_X) \leq 0$. Since there are no holomorphic forms
on a rationally connected variety and the general fibre of the MRC-fibration has
dimension $n-2$, we see that 
\[
h^k(X, \sO_X) = h^0(X, \Omega_X^k)=0 \qquad \forall \ k \geq 3. 
\]
Thus $\chi(X, \sO_X) \leq 0$ implies that $h^1(X, \sO_X) \neq 0$
and we have a non-trivial Albanese morphism 
\holom{\alpha}{X}{Alb(X)}.
We claim that there exists a $j \in \{1, \ldots, n-1\}$ such that 
the direct image sheaf $\alpha_* \sO_X(K_X+jA)$ is not zero:
indeed if $F$ is a general non-empty fibre of $\alpha$, then 
by Proposition \ref{propositionbasic} there exists a  $j \in \{1, \ldots, n-1\}$ such that 
\[
H^0(F, \sO_F(K_F+jA)) \neq 0.
\]
We now argue as in \cite{Xie09}:
let $P \in \Pic0(Alb(X))$ be a numerically trivial Cartier divisor, then
$jA+\alpha^* P$ is nef and big. Using the relative Kawamata-Viehweg theorem
and the Leray spectral sequence one obtains
\[
H^i(Alb(X), \alpha_* \sO_X(K_X+jA) \otimes \sO_{Alb(X)}(P))=0  \qquad \forall \ i>0.
\]
Therefore \cite[Cor.2.4]{Muk81} implies that
\[
H^0(X, \sO_X(K_X+jA+ \alpha^* P))
\simeq 
H^0(Alb(X),  \alpha_* \sO_X(K_X+jA) \otimes \sO_{Alb(X)}(P)) \neq 0
\]
for some $P \in \Pic0(Alb(X))$. In particular 
\[
\chi(X, \sO_X(K_X+jA+P)) \neq 0.
\]
Since tensoring with a numerically trivial Cartier divisor does not change the Euler characteristic, 
we get a contradiction to $\chi(X, \sO_X(K_X+jA))=0$. 

\begin{center}
{\bf  Case IV: $\dim Y \geq 3$.}
\end{center}

The following lemma is due to Fujita \cite[Lemma 2.5]{Fuj87} in the case where $A$ is an ample Cartier divisor and 
$X$ is Gorenstein and to Andreatta in the log-terminal setting \cite[Thm.2.1]{And95}.

\begin{lemma}  \label{lemmafujitabirational} 
Let $X$ be a normal, projective variety of dimension $n$ with at most log-terminal singularities.
Let \holom{\mu}{X}{X'} be an elementary contraction of birational type contracting a $K_X$-negative extremal ray $\Gamma$.
Let $\fibre{\mu}{y}$ be a fibre of dimension $r>0$.

If $A$ is a nef and big Cartier divisor on $X$ such that $A \cdot \Gamma>0$, then
\[
(K_X+r A) \cdot \Gamma \geq 0.
\]
\end{lemma}

In order to simplify the notation we assume without loss of generality that $X$ is smooth, so $X'=X$.
Note that by Proposition \ref{propositiondirectimage} the divisor $K_{X}+(n-2)A$ is pseudoeffective,
in particular $K_{X}+(n-1)A$ is big.

{\em Step 1. Reduction to the case where $K_X+(n-1)A$ is nef and big.}
Our goal is to prove that there exists a birational map $\merom{\psi}{X}{X_{min}}$
onto a projective variety $X_{min}$ with at most terminal singularities and a nef and big
Cartier divisor $A_{min}$ on $X_{min}$ such that $K_{X_{min}}+(n-1)A_{min}$ is nef and
\[
H^0(X_{min}, \sO_{X_{min}}(K_{X_{min}}+jA_{min})) \simeq H^0(X, \sO_{X}(K_{X}+jA)) \qquad \forall \ j \in \{1, \ldots, n-1 \}. 
\] 
We will construct $X_{min}$ by using an appropriate MMP:
since $A$ is nef and big, there exists an effective $\Q$-divisor $D$ on $X$ such that
$D \sim_\Q (n-1)A$ and the pair $(X, D)$ is klt. Since $K_X+D$ is pseudoeffective and $D$ is big
we know by \cite[Thm.1.2]{BCHM06} that
the pair $(X,D)$ has a log-minimal model $(X_{min}, D_{min})$, i.e. we can run a 
$K_X+D$-MMP with scaling
\[
(X_0, D_0):=(X,D) \stackrel{\mu_0}{\rightarrow} (X_1, D_1) \stackrel{\mu_1}{\rightarrow} \ldots
\stackrel{\mu_s}{\rightarrow} (X_s,D_s) =: (X_{min}, D_{min}).
\]
We claim that if 
\merom{\mu_i}{(X_i,D_i)}{(X_{i+1},D_{i+1})} is an elementary contraction contracting an extremal ray $\Gamma_i$ in this MMP, 
then $D_i \cdot \Gamma_i=0$. Moreover one has $D_{i+1} \sim_\Q (n-1) A_{i+1}$ with $A_{i+1}$ a nef and big Cartier divisor such that
$A_i=\mu_i^* A_{i+1}$ and
\[
H^0(X_i, \sO_{X_i}(K_{X_i}+jA_i)) \simeq H^0(X_{i+1}, \sO_{X_{i+1}}(K_{X_{i+1}}+jA_{i+1})) \qquad \forall \ j \in \{1, \ldots, n-1 \}.
\]
In particular the $K_X+D$-MMP is a $K_X$-MMP, so $X_{min}$ has terminal singularities. 
Hence if we set $A_{min}:=A_s$, then $K_{X_{min}}+(n-1)A_{min}$ is nef and our non-vanishing problem descends to
$X_{min}$.

{\em Proof of the claim.}
Since the contraction $\mu_i$ is $K_{X_i}+D_i$-negative, we have
\[
(K_{X_i}+(n-1)A_i) \cdot \Gamma_i = (K_{X_i}+D_i) \cdot \Gamma_i < 0.
\]
Since $\mu_i$ is birational Lemma \ref{lemmafujitabirational} shows that $A_i \cdot \Gamma_i=0$.

a) If the contraction is divisorial, then $\mu_i$ is a morphism and  $K_{X_i}=\mu_i^* K_{X_{i+1}}+E_i$ with $E_i$ an effective $\Q$-divisor.
Since $A_i \cdot \Gamma_i=0$ there exists a nef and big Cartier divisor $A_{i+1}$ on $X_{i+1}$ such that $A_i=\mu_i^* A_{i+1}$. 
Thus we have $\Delta_{i+1} \sim_\Q (n-1)A_{i+1}$
and
\[
H^0(X_i, \sO_{X_i}(K_{X_i}+jA_i)) \simeq H^0(X_{i+1}, \sO_{X_{i+1}}(K_{X_{i+1}}+jA_{i+1})) \qquad \forall \ j \in \{1, \ldots, n-1 \}.
\]

b) If the contraction is small, denote by $\holom{\nu}{X_i}{X'}$ and $\holom{\nu_+}{X_{i+1}}{X'}$ the birational morphisms defining the flip.
Since $A_i \cdot \Gamma_i=0$ there exists a nef and big Cartier divisor $A'$ on $X'$ such that $A_i=\nu^* A'$. 
Thus $A_{i+1}:=\nu^*_+ A'$ is a nef and big Cartier divisor such that $\Delta_{i+1} \sim_\Q (n-1) A_{i+1}$. 
Since $X_i$ and $X_{i+1}$ are isomorphic in codimension one, we have
\[
H^0(X_i, \sO_{X_i}(K_{X_i}+jA_i)) \simeq H^0(X_{i+1}, \sO_{X_{i+1}}(K_{X_{i+1}}+jA_{i+1})) \qquad \forall \ j \in \{1, \ldots, n-1 \}. 
\]

{\em Step 2. The computation.}
We know that $K_{X_{min}}+(n-2)A_{min}$ is pseudoeffective and by the first step $K_{X_{min}}+(n-1)A_{min}$ is nef and big.
The goal of this step is to show that
\[
A_{min}^{n-2} \cdot 
[
2 (K_{X_{min}}^2+c_2(X_{min}))
+
6n A_{min} \cdot K_{X_{min}}  
+
(n+1) (3n-2) A_{min}^2
]
\]
is positive. 
Note first that
\begin{eqnarray*}
& & A_{min}^{n-2} \cdot 
[
2 (K_{X_{min}}^2+c_2(X_{min}))
+
6n A_{min} \cdot K_{X_{min}}  
+
(n+1) (3n-2) A_{min}^2
] 
\\
& = &
2 A_{min}^{n-2} (K_{X_{min}}+(n-1)A_{min}) \cdot (K_{X_{min}}+(n+2)A_{min})
\\
& & \qquad + A_{min}^{n-2} \cdot
\left[
(2n-2) K_{X_{min}} \cdot A_{min}
+
(n^2-n+2) A_{min}^2
+
2 c_2(X_{min})
\right]. 
\end{eqnarray*}
Since $K_{X_{min}}+(n-1)A_{min}$ is nef and big, the first term is positive. Thus we are left to show that 
\[
A_{min}^{n-2} \cdot
\left[
(2n-2) K_{X_{min}} \cdot A_{min}
+
(n^2-n+2) A_{min}^2
+
2 c_2(X_{min})
\right]
\geq
0.
\]

{\em 1st case. $(X_{min},A_{min})$ is not birationally a scroll.}
Then $\Omega_{X_{min}}\hspace{-0.8ex}<\hspace{-0.8ex}A_{min}\hspace{-0.8ex}>$ is generically nef by Theorem \ref{theoremgenericnefCartier}. 
Since  $K_{X_{min}}+(n-1)A_{min}$ is nef, 
$\det \Omega_{X_{min}}\hspace{-0.8ex}<\hspace{-0.8ex}A_{min}\hspace{-0.8ex}>=K_{X_{min}}+nA_{min}$ is nef. 
Since $X_{min}$ is smooth in codimension two we know by Corollary \ref{corollarymiyaokaQ} that 
\[
A_{min}^{n-2}  \cdot c_2(X_{min}) \geq - A_{min}^{n-2} \cdot  
\left(
(n-1) K_{X_{min}} \cdot A_{min} + \frac{(n-1)n}{2} A_{min}^2
\right).
\]
Therefore
\begin{eqnarray*}
& & A_{min}^{n-2} \cdot
\left[
(2n-2) K_{X_{min}} \cdot A_{min}
+
(n^2-n+2) A_{min}^2
+
2 c_2(X_{min})
\right]
\\
& \geq & 
A_{min}^{n-2} \cdot
\left[
(n^2-n+2) A_{min}^2
-
(n-1)n A_{min}^2
\right]
= 2 A_{min}^n \geq 0.
\end{eqnarray*}

{\em 2nd case. $(X_{min},A_{min})$ is birationally a scroll.}

Since $A_{min}$ is a limit of ample $\Q$-Cartier $\Q$-divisors, the problem reduces 
to showing that if $S$
is a surface cut out by general divisors $D_j \in | m_j H_j|$ where the $H_j$ are ample Cartier divisors 
and $m_j \gg 0$, then one has
\[
[S] \cdot
\left[
(2n-2) K_{X_{min}} \cdot A_{min}
+
(n^2-n+2) A_{min}^2
+
2 c_2(X_{min})
\right]
\geq
0.
\]
Note that since $X_{min}$ is smooth in codimension two,
the surface $S$ is smooth. The main difficulty is to estimate $[S] \cdot c_2(X_{min})$ which we will do now.

Denote by $T_{X_{min}}:=\Omega_{X_{min}}^*$ the tangent sheaf of $X_{min}$.
Fix $H_1, \ldots, H_{n-1}$ ample Cartier divisors, 
and let 
\[
0= \sF_0 \subsetneq \sF_1 \subsetneq \ldots \subsetneq \sF_r=T_{X_{min}}
\]
be the Harder-Narasimhan filtration of $T_{X_{min}}$ with respect to $H_1, \ldots, H_{n-1}$. 
Then for $i \in \{ 1, \ldots, r\}$, the graded pieces $\sG_i:=\sF_i/\sF_{i-1}$ are semistable torsion-free sheaves
and if $\mu(\sG_i)$ denotes the slope, we have a strictly decreasing sequence
\[
\mu(\sG_1) > \mu(\sG_2) > \ldots > \mu(\sG_r).
\]
Set $d_i:=\rk \sG_i$, and let $l \in \N$ be as in Lemma \ref{lemmatechnical}, then 
\[
\mu(\sG_i\hspace{-0.8ex}<\hspace{-0.8ex}-A_{min}\hspace{-0.8ex}>) \leq 0 \qquad \forall \ i \geq l+1.
\]
Moreover by Lemma \ref{lemmatechnical},b) there exists 
a sequence of rational numbers $w_1, \ldots, w_l$ such that
\[
w_i \in [d_i, d_i+1] \qquad \forall \ i \in \{ 1, \ldots, l \}
\]
and
\begin{equation} \label{weightsum}
\sum_{i=1}^l w_i = (\sum_{i=1}^l d_i)+1
\end{equation}
and
\[
\mu(\sG_i\hspace{-0.8ex}<\hspace{-0.8ex}-\frac{w_i}{d_i}A_{min}\hspace{-0.8ex}>) \leq 0 \qquad \forall \ i \in \{ 1, \ldots, l \}.
\]
Note furthermore that $\Omega_{X_{min}}$ contains a generically nef subsheaf of rank at least three (the pull-back of the cotangent
sheaf of the base of the MRC-fibration). Thus there exists a $k \in \{ l+1, \ldots, r \}$ such that
\[
\mu(\sG_i) \leq 0 \qquad \forall \ i \geq k
\]
and $\sum_{i=k}^r d_i \geq 3$.
For $i \in \{ 1, \ldots, r\}$ we set  
\[
V_i := \sG^*_i|_S.
\]
Then the $V_j$ are locally free sheaves on the surface $S$.
Since $S$ is a smooth surface (so every ideal sheaf has a locally free resolution of length at most one), 
we have by \cite[Lemma 10.9]{Uta92}
\[
c_2(\sG_i|_S) \geq c_2((\sG_i|_S)^{**})=c_2((\sG_i|_S)^{*})  \qquad \forall \ i \in \{ 1, \ldots, r\}.
\]
Moreover we have $\sG^*_i|_S \simeq (\sG_i|_S)^*$, so we obtain
\[
[S] \cdot c_2(X_{min}) = c_2(T_{X_{min}}|_S)= c_2(\oplus_{i=1}^r \sG_i|_S)
\geq c_2(\oplus_{i=1}^r V_i).
\]

Our goal is to estimate $c_2(\oplus_{i=1}^r V_i)$ by applying Lemma \ref{lemmamiyaokaQ}
to a sufficiently positive $\Q$-twist. For $i \in \{ l+1, \ldots, k-1\}$ we set
\[
w_i := d_i
\]
and for $i \in \{ k, \ldots, r\}$ we set
\[
w_i := 0.
\]
With this notations the slope estimates imply that for all $i \in \{ 1, \ldots, r\}$ the twisted vector bundle
$V_i\hspace{-0.8ex}<\hspace{-0.8ex}\frac{w_i}{d_i}A_{min}\hspace{-0.8ex}>$
is generically nef\footnote{In order to simplify
the notation we denote by $A_{min}$ the restriction of $A_{min}$ to $S$.}.
By Formula \eqref{qc1} we have
\[
c_1(V_i\hspace{-0.8ex}<\hspace{-0.8ex}\frac{w_i}{d_i}A_{min}\hspace{-0.8ex}>)
=
c_1(V_i)+w_i A_{min} \qquad \forall \ i \in \{ 1, \ldots, r\}.
\]
Since $V_i\hspace{-0.8ex}<\hspace{-0.8ex}\frac{w_i}{d_i}A_{min}\hspace{-0.8ex}>$ is generically nef and $A_{min}$ is nef, 
this implies
\begin{equation} \label{equationc1estimate}
c_1(V_i) \cdot A_{min} \geq - w_i A_{min}^2 \qquad \forall \ i \in \{ 1, \ldots, r\}.
\end{equation}
Since $\sum_{i=1}^l w_i = (\sum_{i=1}^l d_i)+1$ and
$\sum_{i=l+1}^{k-1} w_i = \sum_{i=l+1}^{k-1} d_i$, we have
\[
\sum_{i=1}^{r} w_i = (\sum_{i=1}^{k-1} d_i)+1=n-(\sum_{i=k}^r d_i)+1 \leq n-2.
\]
Thus if we set
\[
c := \frac{n-1}{\sum_{i=1}^{r} w_i},
\]
we have $c \geq 1$ and $\sum_{i=1}^r c w_i=n-1$. Since $K_{X_{min}}|_S= \sum_{i=1}^r c_1(V_i)$ the twisted vector bundle
\[
\bigoplus_{i=1}^r V_i\hspace{-0.8ex}<\hspace{-0.8ex}\frac{c w_i}{d_i}A_{min}\hspace{-0.8ex}>
\]
is generically nef with nef determinant $(K_{X_{min}}+(n-1)A_{min})|_S$. 
Thus its second Chern class is non-negative by Lemma \ref{lemmamiyaokaQ}, so
by Lemma \ref{lemmac2formula} below we have
\begin{eqnarray*}
0 & \leq & c_2( \oplus_{i=1}^r V_i\hspace{-0.8ex}<\hspace{-0.8ex}\frac{c w_i}{d_i}A_{min}\hspace{-0.8ex}> )
\\
& = &
c_2( \oplus_{i=1}^r V_i )
+
\frac{1}{2} 
\left(
(\sum_{i=1}^r c w_i)^2-\sum_{i=1}^r \frac{(c w_i)^2}{d_i}
\right) A_{min}^2
+ 
\sum_{i=1}^r 
\left(
(\sum_{j=1}^r c w_j) - \frac{c w_i}{d_i}
\right)
c_1(V_i) \cdot A_{min}. 
\end{eqnarray*}
Since $\sum_{j=1}^r c w_j=n-1$ and $K_{X_{min}}|_S= \sum_{i=1}^r c_1(V_i)$ we have
\[
\sum_{i=1}^r 
\left(
(\sum_{j=1}^r c w_j) - \frac{c w_i}{d_i}
\right)
c_1(V_i) \cdot A_{min}
=
(n-1) K_{X_{min}}|_S \cdot A_{min} - \sum_{i=1}^r \frac{c w_i}{d_i} c_1(V_i) \cdot A_{min}.
\]
By inequality \eqref{equationc1estimate} this is less or equal than
\[
(n-1) K_{X_{min}}|_S \cdot A_{min}+\sum_{i=1}^r \frac{c w_i^2}{d_i} A_{min}^2.
\]
Thus we get
a lower bound for the second Chern class:
\begin{equation}
c_2(\oplus_{i=1}^r V_i)  \geq
- (n-1) K_{X_{min}}|_S \cdot A_{min} 
-
\frac{1}{2} 
\left[ 
(n-1)^2+(2c-c^2) \sum_{i=1}^r \frac{w_i^2}{d_i}
\right]
A_{min}^2.
\end{equation}
This immediately implies that
$[S] \cdot [(2n-2) K_{X_{min}} \cdot A_{min}
+
(n^2-n+2) A_{min}^2
+
2 c_2(X_{min}) ]$ is greater or equal than
\[
\left[ (n^2-n+2)
-(n-1)^2
-(2c-c^2) \sum_{i=1}^r \frac{w_i^2}{d_i}
\right] [S] \cdot A_{min}^2
=
\left[ n+1
-(2c-c^2) \sum_{i=1}^r \frac{w_i^2}{d_i}
\right] [S] \cdot A_{min}^2
.
\]
We have $2c-c^2 \leq 1$ so we are finished if we show that
\[
\sum_{i=1}^r \frac{w_i^2}{d_i} \leq n+1.
\]
Recall now that by Equation \eqref{weightsum} we have
$\sum_{i=1}^l (w_i-d_i) = 1$,
so 
\[
\sum_{i=1}^l \frac{w_i^2}{d_i}= \sum_{i=1}^l \frac{d_i^2+2(w_i-d_i)d_i+(w_i-d_i)^2}{d_i}
= \sum_{i=1}^l d_i + 2 + \sum_{i=1}^l \frac{(w_i-d_i)^2}{d_i}.
\]
Since $(w_i-d_i) \in [0, 1]$ and $d_i \geq 1$ we have moreover
\[
\sum_{i=1}^l \frac{(w_i-d_i)^2}{d_i} \leq \sum_{i=1}^l (w_i-d_i) = 1.
\]
Since $w_i=d_i$ for $i \in \{ l+1, \ldots, k-1\}$ and $w_i=0$ for $i \in \{ k, \ldots, r\}$ we get
\[
\sum_{i=1}^r \frac{w_i^2}{d_i} \leq \sum_{i=1}^{k-1} d_i +3=n-\sum_{i=k}^{r} d_i+3.
\]
Since $\sum_{i=k}^{r} d_i \geq 3$ this finishes the proof of this step.

{\em Step 3. The conclusion.}
Let \holom{\mu}{X_{min}'}{X_{min}} be a desingularisation of $X_{min}$. 
Since $X_{min}$ is smooth in codimension two one has
\[
(\mu^* A_{min})^{n-2} \cdot
\left[
2 (K_{X_{min}'}^2+c_2(X_{min}'))
+
6n \mu^* A_{min} \cdot K_{X_{min}'} 
+
(n+1) (3n-2) (\mu^*A_{min})^2 
\right]
\]
\[
=A_{min}^{n-2} \cdot
[
2 (K_{X_{min}}^2+c_2(X_{min}))
+
6n A_{min} \cdot K_{X_{min}}  
+
(n+1) (3n-2) A_{min}^2]
\]
which is positive by Step 2.
Since terminal singularities are rational, we have
\[
\chi(X_{min}',  \sO_{X_{min}'}(K_{X_{min}'}+j\mu^*A_{min})) = \chi(X_{min}, \sO_{X_{min}}(K_{X_{min}}+jA_{min}))=0
\]
for all $j=1, \ldots, n-1$. 
Thus by Lemma \ref{lemmaminusone} the Equation \eqref{equationtwo} holds, a contradiction to our computation.
\end{proof}

\begin{lemma} \label{lemmac2formula}
Let $S$ be a projective manifold. Let $V_1, \ldots, V_r$ be vector bundles on $S$, 
and let $A$ be a Cartier divisor class on $S$. Set $d_i:=\rk V_i$,
and let $\alpha_i \in \Q$ for $i \in \{1, \ldots, r\}$. Then $c_2( \oplus_{i=1}^r V_i\hspace{-0.8ex}<\hspace{-0.8ex}\frac{\alpha_i}{d_i}A\hspace{-0.8ex}> )$ is equal to 
\[
c_2( \oplus_{i=1}^r V_i )
+
\frac{1}{2} 
\left(
(\sum_{i=1}^r \alpha_i)^2-\sum_{i=1}^r \frac{\alpha_i^2}{d_i}
\right) A^2
+ 
\sum_{i=1}^r 
\left(
(\sum_{j=1}^r \alpha_j) - \frac{\alpha_i}{d_i}
\right) 
c_1(V_i) \cdot A.
\] 
\end{lemma}

\begin{proof}
Note first that by Formula \eqref{qc2} one has
\begin{equation} \label{equationb}
c_2(V_i\hspace{-0.8ex}<\hspace{-0.8ex}\frac{\alpha_i}{d_i}A\hspace{-0.8ex}>)
=
c_2(V_i)+\frac{1}{2} (\alpha_i^2-\frac{\alpha_i^2}{d_i}) A^2+(\alpha_i-\frac{\alpha_i}{d_i}) c_1(V_i) \cdot A.
\end{equation}
Recall also that for a direct sum of ($\Q$-twisted) vector bundles  $\oplus_{i=1}^r \sF_i$ one has
\[
c_2(\oplus_{i=1}^r \sF_i) = \sum_{i=1}^r c_2(\sF_i) + \sum_{i<j} c_1(\sF_i) \cdot c_1(\sF_j).
\]
Thus 
$$
c_2( \oplus_{i=1}^r V_i\hspace{-0.8ex}<\hspace{-0.8ex}\frac{\alpha_i}{d_i}A\hspace{-0.8ex}> )
=
\sum_{i=1}^r c_2(V_i\hspace{-0.8ex}<\hspace{-0.8ex}\frac{\alpha_i}{d_i}A\hspace{-0.8ex}>) 
+ \sum_{i<j} c_1(V_i\hspace{-0.8ex}<\hspace{-0.8ex}\frac{\alpha_i}{d_i}A\hspace{-0.8ex}>) 
\cdot
c_1(V_j\hspace{-0.8ex}<\hspace{-0.8ex}\frac{\alpha_j}{d_j}A\hspace{-0.8ex}>)
$$
which by \eqref{equationb} and Formula \eqref{qc1} is equal to
\[
\sum_{i=1}^r 
\left(
c_2(V_i)+\frac{1}{2} (\alpha_i^2-\frac{\alpha_i^2}{d_i}) A^2+(\alpha_i-\frac{\alpha_i}{d_i}) c_1(V_i) A
\right)
+ \sum_{i<j} (c_1(V_i)+\alpha_i A) \cdot (c_1(V_j)+\alpha_j A)
\]
\[
= 
c_2(\oplus_{i=1}^r V_i) + \frac{1}{2} 
\left(
\sum_{i=1}^r \alpha_i^2 + 2 \sum_{i<j} \alpha_i \alpha_j - \sum_{i=1}^r \frac{\alpha_i^2}{d_i}
\right) A^2
+ \sum_{i=1}^r (\alpha_i-\frac{\alpha_i}{d_i}) c_1(V_i) \cdot A + \sum_{i<j} (\alpha_j c_1(V_i) \cdot A+\alpha_i c_1(V_j) \cdot A).
\]
By the binomial formula the coefficient for $A^2$ equals $(\sum_{i=1}^r \alpha_i)^2-\sum_{i=1}^r \frac{\alpha_i^2}{d_i}$,
so we are left to show that
\[
\sum_{i=1}^r \alpha_i c_1(V_i) \cdot A + \sum_{i<j} (\alpha_j c_1(V_i) \cdot A+\alpha_i c_1(V_j) \cdot A)
=
\sum_{i=1}^r 
(\sum_{j=1}^r \alpha_j)  
c_1(V_i) \cdot A.
\]
This is elementary by induction on $r$.
\end{proof}

\section{The Ambro-Ionescu-Kawamata conjecture} 

\begin{proof}[Proof of Theorem \ref{theoremIK}]
By \cite[Thm.3.1]{Kaw00} we can suppose that $K_X+A$ is nef and big.

{\em Step 1. Terminalisation.}
By \cite[Thm.6.23]{KM98} there exists a terminalisation of $X$, i.e
a birational map \holom{\mu}{X'}{X} from a threefold with at most terminal singularities such that $K_{X'}=\mu^*K_X$.
Thus $A':=\mu^* A$ is a nef and big Cartier divisor such that $K_{X'}+A'$ is nef and big, moreover we have
\[
H^0(X, \sO_X(K_X+A)) = H^0(X', \sO_{X'}(K_{X'}+A')).
\]
Hence the non-vanishing problem lifts to $X'$, in 
order to simplify the notation
we suppose without loss of generality that $X$ has at most terminal singularities. 

{\em Step 2. The computation.}
We claim that the twisted cotangent sheaf $\Omega_X\hspace{-0.8ex}<\hspace{-0.8ex}\frac{2}{3}A\hspace{-0.8ex}>$ is generically 
nef. 
Assuming this for the time being, let us show how to conclude. By \cite[p.541]{Kaw86}, we have
\[
\chi(X, \sO_X) \geq \frac{-1}{24} K_X \cdot c_2(X).
\]
By the Riemann-Roch formula for threefolds with terminal singularities \cite[p.413]{Rei87}
\[
\chi(X, \sO_X(K_X+A)) \geq 
\frac{1}{12} (K_X+A)\cdot A \cdot (K_X+2A) +  \frac{1}{24} (K_X+2A) \cdot c_2(X). 
\]
Since $\Omega_X\hspace{-0.8ex}<\hspace{-0.8ex}\frac{2}{3}A\hspace{-0.8ex}>$ is generically nef and $K_X+2A$ is nef, we have
by Corollary \ref{corollarymiyaokaQ}
\[
(K_X+2A) \cdot c_2(X) \geq - (K_X+2A) \cdot (\frac{4}{3} K_X \cdot A + \frac{4}{3} A^2) = \frac{4}{3} (K_X+2A) \cdot (K_X+A) \cdot A.
\]
Hence
\[
\frac{1}{12} (K_X+A) \cdot  A \cdot (K_X+2A) +  \frac{1}{24} (K_X+2A) \cdot c_2(X)
\geq
\frac{1}{24} (K_X+2A) \cdot  (K_X+A) \cdot
\frac{2}{3} A.
\]
Since the three divisors $K_X+A, A$ and $K_X+2A$ are nef and big this intersection product is strictly positive.
Thus by Kawamata-Viehweg vanishing
\[
h^0(X, \sO_X(K_X+A)) = \chi(X, \sO_X(K_X+A)) > 0.
\]

{\em Proof of the claim.}
We argue by contradiction. Then by Theorem \ref{theoremgenericnef}
there exists a birational morphism \holom{\mu}{X'}{X} and a fibration \holom{\varphi}{X'}{Y} 
such that the general fibre $F$ satisfies
\begin{equation} \label{equationvanishing}
H^0(F, \sO_F(D)) = 0
\end{equation}
where $D$ is a Cartier divisor on $F$ such that $D \sim_\Q K_F+ \frac{2j}{3} \mu^*A$ with $j \in [0, 3-\dim Y]~\cap~\Q$.
Since $X$ is terminal, we have
\[
K_{X'}= \mu^*K_{X}+E
\]
for some effective $\Q$-divisor $E$. Since $K_X+A$ is nef, this implies that $K_{X'}+\mu^* A$ is pseudoeffective.

{\em 1st case. $\dim Y=1$.} Since $K_{X'}+\mu^*A$ is pseudoeffective, the 
restriction to a general fibre $K_F+\mu^*A|_F$ is pseudoeffective.
Moreover by Equation \ref{equationvanishing} one has
\[
H^0(F, \sO_F(K_F+ \mu^* A))=0.
\]
This contradicts Theorem \ref{theoremBS}. 

{\em 2nd case. $\dim Y=2$.}
Let $F \simeq \PP^1$ be a general $\varphi$-fibre, then
$$
(K_X+A) \cdot \mu(F) > 0
$$
since $K_X+A$ is big. 
Thus 
\[
(K_{X'}+\mu^* A) \cdot F = (\mu^*(K_{X}+A)+E) \cdot F > 0.
\]
Since $K_{X'} \cdot F=-2$ and $A$ is Cartier, this implies that $\mu^* A \cdot F \geq 3$. 
Hence $K_F+ \frac{2}{3} A|_F$ is $\Q$-linear equivalent to an effective divisor, 
a contradiction to \eqref{equationvanishing}.
\end{proof}

\end{document}